\documentclass[12pt]{amsart}

\usepackage{tikz}
\usetikzlibrary{snakes}
\usepackage[colorlinks=true, linkcolor=blue, anchorcolor=blue, citecolor=blue, filecolor=blue, menucolor= blue, urlcolor=blue]{hyperref}
\usepackage{amsmath,amssymb,amsthm,amscd}
\usepackage{verbatim}
\usepackage{comment}
\usepackage{multirow}
\usepackage{mathtools}
\usepackage{enumitem}
\usepackage{pgf,tikz}
\usetikzlibrary{arrows}
\usepackage[normalem]{ulem}

\usepackage[margin=1in]{geometry}

\numberwithin{equation}{section}

\newtheorem{theorem}{Theorem}[section]
\newtheorem{proposition}[theorem]{Proposition}
\newtheorem{lemma}[theorem]{Lemma}

\newtheorem*{theorem*}{Theorem}

\theoremstyle{definition}
\newtheorem{definition}[theorem]{Definition}
\newtheorem{notation}[theorem]{Notation}
\newtheorem{example}[theorem]{Example}
\newtheorem{remark}[theorem]{Remark}

\newcommand{\Tor}{\ensuremath{\mathrm{Tor}}\hspace{1pt}}

\definecolor{MyDarkGreen}{cmyk}{0.7,0,1,0}

\def\cocoa{{\hbox{\rm C\kern-.13em o\kern-.07em C\kern-.13em o\kern-.15em A}}}

\begin{document}

\title[Minimal Resolution Conjecture on Quartics]{The Minimal Resolution Conjecture on a general quartic surface in $\mathbb P^3$}

\author[M.\ Boij]{M. Boij}
\address{Department of Mathematics, KTH
Royal Institute of Technology, S-100 44 Stockholm, Sweden}
\email{boij@kth.se}

\author{J.\ Migliore}
\address{Department of Mathematics \\
University of Notre Dame \\
Notre Dame, IN 46556 USA}
 \email{migliore.1@nd.edu}

\author[R.\ M.\ Mir\'o-Roig]{R.M.\ Mir\'o-Roig}
\address{Facultat de Matem\`atiques, Departament de Matem\`atiques i Inform\`atica,
Gran Via des les Corts Catalanes 585, 08007 Barcelona, Spain}
\email{miro@ub.edu}

\author{U.\ Nagel}
\address{Department of Mathematics\\
University of Kentucky\\
715 Patterson Office Tower\\
Lexington, KY 40506-0027 USA}
\email{uwe.nagel@uky.edu}

\begin{abstract}
Musta\c{t}\u{a} has given a conjecture for the graded Betti numbers in the minimal free resolution of the ideal of a general set of points on an irreducible projective algebraic variety. For surfaces in $\mathbb P^3$ this conjecture has been proven for points on quadric surfaces and on general cubic surfaces. In the latter case, Gorenstein liaison was the main tool. Here we prove the conjecture for general quartic surfaces. Gorenstein liaison continues to be a central tool, but to prove the existence of our links we make use of certain dimension computations. We also discuss the higher degree case, but now the dimension count does not force the existence of our links.
\end{abstract}

%\date{March 16, 2017}

\thanks{
{\bf Acknowledgements}:
Boij was partially supported by the grant VR2013-4545.
Migliore was partially supported by Simons Foundation grant \#309556.
Mir\'o-Roig was partially supported by MTM2016-78623-P.
Nagel was partially supported by Simons Foundation grant \#317096.
This paper resulted from work done during a Research in Pairs stay at the Mathematisches Forschungsinstitut Oberwolfach in 2017. All four authors are grateful for the stimulating atmosphere and the generous support of the MFO, as well as to the referee for insightful comments}



\maketitle

%\tableofcontents

%%%%%%%%%%%%%%%%%%%%%%%%%%%%%%%%%%%%%%%%%%%%%%%%%%%%%%%%%%%%%%

\section{Introduction}

The shape of the minimal free resolution (MFR) of a general set, $X$, of points in the projective plane has been known for many years, probably due to Gaeta. See \cite{gaeta} for a more recent description by the same author. For points in $\mathbb P^3$ the shape of the MFR was discovered and shown by Ballico and Geramita \cite{BG}. In her Ph.D. thesis and later in \cite{lorenzini}, Lorenzini conjectured the shape of the MFR for a general set of points in $\mathbb P^n$. Roughly speaking,  if we denote by $R$ the polynomial ring $k[x_0,\dots,x_n]$ ($k$ is an algebraically closed field) and we set
\[
\beta_{i,j} = \hbox{Tor}^i (R/I_X,k)_{i+j},
\]
then the Betti diagram $\{ \beta_{i,j} \}$ consists of two non-trivial rows and we have $\beta_{i,j} \cdot \beta_{i+1,j} = 0$ for all $i$ and $j$. The latter condition says that there are no redundant terms in the MFR.
This conjecture was proven for points in $\mathbb P^4$ by Walter \cite{W}, and for large numbers of points in any projective space by Hirschowitz and Simpson \cite{HS}. However, Schreyer experimentally found a probable counterexample in the case of 11 general points in $\mathbb P^6$ (and a few others), and in a stunning development, a proof was given by Eisenbud and Popescu \cite{EP} that these were, in fact, counterexamples and that there is a much larger class of counterexamples. Further work in this direction was done by Eisenbud, Popescu, Schreyer and Walter \cite{EPSW}.

A new Minimal Resolution Conjecture (MRC) was formulated by Musta\c{t}\u{a} in \cite{mustata} concerning a general set, $X$, of sufficiently many points on an irreducible algebraic subvariety, $S$, of $\mathbb P^n$. Essentially the MRC says that the top part of the Betti diagram for $R/I_X$ consists of the Betti diagram for $R/I_S$, and that below this part there are only two nonzero rows, again with no redundant terms in the MFR. In that paper he stated the MRC and proved several initial results, including a nice periodicity theorem in the case where $S$ is an integral curve. Later, Farkas, Musta\c{t}\u{a} and Popa \cite{FMP} proved the MRC for the case where $S$ is a nonhyperelliptic canonically embedded curve. On the other hand, they proved that the MRC fails when $S$ is a curve of large degree.

Turning to the case that $S$ is a surface in $\mathbb P^3$, several cases of the MRC have been proven. Giuffrida, Maggioni and Ragusa \cite{GMR} proved it for a general set of points on a smooth quadric surface in $\mathbb P^3$. Casanellas \cite{casanellas} proved it for certain cardinalities of points on a smooth cubic surface, and the result was extended to any number of points on a smooth cubic surface more or less at the same time by Migliore and Patnott \cite{MP} and by Mir\'o-Roig and Pons-Llopis \cite{MRP2}. The same latter two authors also prove the MRC for certain cardinalities on del Pezzo surfaces \cite{MRP1}. Casanellas introduced the use of Gorenstein liaison theory in this question, and the latter two papers continued this approach.

The conjecture includes, as special cases, the so-called Ideal Generation Conjecture  and the Cohen-Macaulay Type Conjecture. These govern the expected number of minimal generators of $I_X$ and the expected Cohen-Macaulay type of $R/I_X$; in other words, they govern the beginning and the end of the MFR. For general points in $\mathbb P^2$, either the Ideal Generation Conjecture or the Cohen-Macaulay Type Conjecture imply the other. For general points in $\mathbb P^3$, knowing both the Ideal Generation Conjecture and the Cohen-Macaulay Type Conjecture gives the MRC as a consequence. For higher projective spaces, this is not true. In \cite{MRP1}, Mir\'o-Roig and Pons-Llopis also address the latter two conjectures for zero-dimensional schemes on del Pezzo surfaces. In \cite{MRP2} the authors also use the known work on ideals generated by general forms of fixed degree in $k[x,y,z]$.

The goal in this paper is to understand the minimal free resolution of a general set of points, $X$, of fixed cardinality on a general surface, $S$, of degree $d$. We will assume that the socle degree of $X$ is larger than $d$, so that $S$ is the unique surface of degree $d$ containing $X$. By semicontinuity, it is enough to produce one set of points with the desired resolution, on any surface of degree $d$. Our strategy is to use Gorenstein liaison to control the resolutions, extending the known methods.  We develop a general framework and then specialize to quartics.

A general set of points on an irreducible surface $S$ is {\em relatively compressed} (see Definition \ref{rel comp}), and this allows us to partition the cardinalities of the general finite subsets according to the {\em socle degree}, $e$, of their artinian reductions. Our approach to the MRC is along the lines of this partition, using induction on $e$. We use liaison theory to describe certain links that, if they exist, are enough to give the conjectured resolution for any given $e \geq d+1$. (Lower values of $e$ come for free from work of Ballico and Geramita \cite{BG}.) This is done in Section~\ref{liaison section}.

%It is known by work of Beauville \cite{beauville} that the equation of a general surface of degree $d \leq 15$ can be written as the Pfaffian of a skew-symmetric matrix of linear forms, and by work of Faenzi \cite{faenzi} that when $d \leq 14$ is even, $S$ can also be written as the Pfaffian of a skew-symmetric matrix of quadratic forms. This is necessary in order for our Gorenstein links to exist, but it is not sufficient.

Producing the needed links turns out to be a very delicate task. In the first version of this paper 
we used a dimension argument based on a theorem of Kleppe and the third author in  \cite{KM},
%(see Theorem \ref{KM formula}), 
which only worked for the case of general quartic surfaces. Although we believe that all the steps of the argument are correct, we were unable to justify one key step pointed out by the referee. Thus we were forced to find a completely different proof, which is contained in Section 4. Unfortunately it also only works for $d=4$ (see Remark \ref{higher degree remark}, although it is possible that suitable modifications will allow progress). Nevertheless, it is enough to give us our main theorem, Theorem \ref{MRC on quartics}, that the MRC is true on a general quartic surface.

The main ideas of the proof are as follows. First, by semicontinuity it is enough to prove the MRC on any quartic surface. Thus, we produce a special quartic surface containing certain arithmetically Cohen-Macaulay (ACM) curves. These curves allow us to produce other ACM curves via liaison, and the Gorenstein links that we need are produced by using twisted anticanonical divisors \cite{KMMNP}, controlling the dimension of certain linear systems, and showing that the links can be performed for general sets of points according to the results in Section~\ref{liaison section}. The argument is by induction on the socle degree, so we also need to produce the initial cases. This is also done largely via liaison, but we also use in a crucial way classical results of Ellingsrud \cite{E} and of Perrin \cite{P}.

%%%%%%%%%%%%%%%%%%%%%%%%%%%%%%%%%%%%%%%%%%%%%%%%%%%%%%%

\section{Background}

We let $R = k[x_0,x_1,x_2,x_3]$ where  $k$ is an algebraically closed field of characteristic zero.
If $R/I$ is a standard graded algebra over $R$, it admits a minimal free resolution
\[
\mathbb F: \ \ 0 \rightarrow \mathbb F_4 \rightarrow \mathbb F_3 \rightarrow \mathbb F_2 \rightarrow \mathbb F_1 \rightarrow R \rightarrow R/I \rightarrow 0.
\]
If $\mathbb F_i = \bigoplus_{j \in \mathbb N} R(-j)^{\beta_{i,j}}$ then $\beta_{i,j} = \dim \Tor _i(\mathbb F,k)_j$ and
the {\em  Betti table} for $R/I$ is the array
\[
\left [
\begin{array}{cccccccccc}
1 & \beta_{1,1} & \beta_{2,2} & \beta_{3,3} & \beta_{4,4} \\
- & \beta_{1,2} & \beta_{2,3} & \beta_{3,4} & \beta_{4,5} \\
- & \beta_{1,3} & \beta_{2,4} & \beta_{3,5} & \beta_{4,6} \\
\vdots & \vdots & \vdots & \vdots  & \vdots
\end{array}
\right ]
\]
(In general, when $R = k[x_0,\dots,x_n]$ then there are $n+2$ columns in the Betti table by Hilbert's theorem.)
Let $X \subset \mathbb P^3$ be a general set of points. The Minimal Resolution Conjecture (MRC) of Lorenzini \cite{lorenzini} was given for points in $\mathbb P^n$, and  says that there are only two non-zero rows of the associated Betti table, and that in these two rows it always holds that $\beta_{i,j} \cdot \beta_{i+1,j} = 0$. As already mentioned, this was shown for $n=3$ by Ballico and Geramita \cite{BG}.

Musta\c t\u a's version of the MRC deals with a general set of points, $X$, on a fixed variety $V$ in projective space. It says, essentially, that for a large set of points the top part of the Betti table is the Betti table of $R/I_V$, and that below this there are again at most two rows, which again satisfy $\beta_{i,j} \cdot \beta_{i+1,j} = 0$. In this paper we will take $V$ to be a general surface, $S$, of degree $d$ in $\mathbb P^3$, and we will state the conjecture more precisely in a moment.

As noted in the introduction, the Minimal Resolution Conjecture is known for sets of points on a smooth quadric and on a general cubic surface in $\mathbb P^3$. Thus from now on in this paper, we will assume that the degree of our surface is $d \geq 4$.
We will use the following notation.

\begin{notation}
Let $S$ be an irreducible surface of degree $d \geq 4$ in $\mathbb P^3$. We denote by $H_S$ the Hilbert function of $S$, and by $h_S$ the first difference $h_S (x)= \Delta H_S (x):= H_S(x) - H_S(x-1)$. For any set of points $X$ on $S$ we define the {\em $h$-vector} of $X$ as the first difference of its Hilbert function.
\end{notation}

\begin{remark} \label{formulas}
We note that
\[
\begin{array}{rcll}
H_S(x) & = & \displaystyle  \binom{x+3}{3} - \binom{x-d+3}{3} \\ \\
& = & \displaystyle \frac{d}{2} x^2 - \left ( \frac{d^2-4d}{2} \right ) x + \binom{d-1}{3} +1 & \hbox{for } x \geq d -1
\end{array}
\]
and
\[
h_S(x)  =  \displaystyle dx - \binom{d-1}{2} +1  \hbox{\ \ \ \ for } x \geq d-1.
\]
Of course the latter is the Hilbert polynomial of a plane curve of degree $d$.
\end{remark}

\begin{definition} \label{rel comp}
Let $S$ be an irreducible surface in $\mathbb P^3$ and let $X \subset S$ be a finite set of  points. We will say that the Hilbert function of $X$ on $S$ is {\em relatively compressed}  if one of the following two situations holds.

\begin{itemize}
\item[(a)] If $X$ is arithmetically Gorenstein of socle degree $n$ with $h$-vector $h_X$ then $h_X (x) = h_S (x)$ for $0 \leq x \leq \left \lfloor \frac{n}{2} \right \rfloor$, and the rest of the $h$-vector is determined by symmetry.

\item[(b)] Otherwise, $X$ is relatively compressed on $S$ if there  are integers $e \geq 0$ and $1 \leq t \leq h_S(e)$ such that
\[
h_X(x) =
\left \{
\begin{array}{cl}
h_S(x) & \hbox{for } 0 \leq x \leq e-1; \\
t & \hbox{for } x = e; \\
0 & \hbox{for } x > e.
\end{array}
\right.
\]
\end{itemize}
We usually simply say that $X$ itself is relatively compressed if its Hilbert function is.
\end{definition}

In particular, a general set of points on $S$ is relatively compressed. The following is an equivalent formulation and additional terminology.

\begin{definition}
Let $S$ be an irreducible surface of degree $d$ in $\mathbb P^3$. Let $X \subset S$ be a general set of points of fixed cardinality. Define the integers $e$ and $t$ by $H_S(e-1) < |X| = H_S(e-1) +t   \leq H_S(e)$ for some $e \geq 3$ and $1 \leq t  \leq h_S(e)$. Then $e$ is called the {\em socle degree} of $X$ and we will call $t$ the {\em surplus} of $X$. We will  denote by $X_{e,t}$ a general set of points on $S$ with socle degree $e$ and surplus $t$. Note that the $h$-vector of $X_{e,t}$ is
\[
\left ( 1,3,6, \dots, h_S(e-1),t \right )
\]
and that
\[
|X_{e,t}| = H_S(e-1)+t.
\]
\end{definition}

We now recall the {\em Minimal Resolution Conjecture (MRC) for surfaces in $\mathbb P^3$}. Let $S$ be a general surface of degree $d$. If $e \leq d$, the conjecture coincides with the MRC (now a theorem \cite{BG}) for points in $\mathbb P^3$.  Now assume that $e \geq d+1$. In this case the conjecture says that the shape of the Betti table for $R/I_{X_{e,t}}$ is
{ \footnotesize \[
\begin{array}{cccccccccc}
1 & -  & - & -  \\
-& 0 & 0 & 0 \\
&& \vdots \\
-& 0 & 0 & 0 \\
- & 1 & - & -  \\
-& 0 & 0 & 0 \\
& & \vdots \\
-& 0 & 0 & 0 \\
- & \beta_{1,e} & \beta_{2,e+1} & \beta_{3,e+2}  \\
- & \beta_{1,e+1} & \beta_{2,e+2} & \beta_{3,e+3} &  \\
- & 0 & 0 & 0 \\
\end{array}
\]}
\hspace{-.3cm} with $\beta_{1,e+1} \cdot \beta_{2,e+1} = 0$ and $\beta_{2,e+2} \cdot \beta_{3,e+2} = 0$. To more easily visualize the needed mapping cones in Section \ref{liaison section}, we will use the following  notation for this conjectured minimal free resolution.
\[
0 \rightarrow
\begin{array}{c}
R(-e-3)^{c_2} \\
\oplus \\
R(-e-2)^{c_1}
\end{array}
\rightarrow
\begin{array}{c}
R(-e-2)^{b_2} \\
\oplus \\
R(-e-1)^{b_1}
\end{array}
\rightarrow
\begin{array}{c}
R(-e-1)^{a_2} \\
\oplus \\
R(-e)^{a_1}\\
\oplus \\
R(-d)
\end{array}
\rightarrow I_{X_{e,t}} \rightarrow 0
\]
where $a_2 \cdot b_1 = 0$, $b_2 \cdot c_1 = 0$, $a_1 = h_S(e)-t$ and  $c_2 = t$. More precisely, we will say that the {\em Ideal Generation Conjecture} holds if $a_2 \cdot b_1 = 0$ and we will say that the {\em Cohen-Macaulay Type Conjecture} holds if $b_2 \cdot c_1 = 0$.

Musta\c t\u a showed in \cite{mustata}  (Examples 1 and 2) that the MRC holds when $t = h_S(e)$ or $t = h_S(e)-1$.

\begin{proposition} \label{intervals}
Let $S$ be an irreducible surface of degree $d$ in $\mathbb P^3$. For a given socle degree $e \geq d$, there are at most four values of the surplus $t$ needed in order to prove the MRC for all general sets of points on $S$ with socle degree $e$. More precisely, let
\[
\begin{array}{rcl}
m_1(e) & = & \max \{ t \ | \ 3t \leq h_S(e-1) \} \\
m_2(e) & = & \min \{ t \ | \ 3t \geq h_S(e-1) \} \\
m_3(e) & = & \max \{ t \ | \ 3(h_S(e) -t) \geq h_S(e+1) \} \\
m_4(e) & = & \min \{ t \ | \ 3(h_S(e) - t) \leq h_S(e+1) \}
\end{array}
\]
Notice that $m_1(e) \leq m_2(e) \leq m_3(e) \leq m_4(e)$.

\begin{itemize}

\item[(a)] If $X_{e,m_1(e)}$ and $X_{e,m_2(e)}$ both satisfy the Cohen-Macaulay Type Conjecture then so does $X_{e,t}$ for all $t$.

\item[(b)] If $X_{e,m_3(e)}$ and $X_{e,m_4(e)}$ both satisfy the Ideal Generation Conjecture then so does $X_{e,t}$ for all $t$.

\item[(c)] If $d$ is divisible by 3 then $m_1(e) = m_2(e)$ and $m_3(e) = m_4(e)$.
\end{itemize}
\end{proposition}

\begin{proof}
This has been observed before -- see for instance \cite{mustata} Proposition 1.7(i) or \cite{MP} section 3 (for the case $d=3$). The idea is to pass to the artinian reduction of $R/I_{X_{e,t}}$. For $t \leq m_1(e)$ the generators of least degree of the canonical module have no linear syzygies, and for $t \geq m_2(e)$ the canonical module is generated in its least degree. Similarly, for $t \leq m_3(e)$ the ideal is generated only in degree~$e$, and for $t \geq m_4(e)$ the generators of degree $e$ in the ideal have no linear syzygies.

\end{proof}

%%%%%%%%%%%%%%%%%%%%%%%%%%%%%%%%%%%%%%%%%%%%%%%%%%%%%%%%%%%%%%

\section{Liaison considerations} \label{liaison section}

In this paper we will make extensive use of Gorenstein liaison.  The arithmetically Gorenstein sets of points on our surface $S$ of degree $d$ that we will consider are {\em relatively compressed}. In the case of even socle degree there is one peak in the $h$-vector, and in the case of odd socle degree there are two.

In this section we will assume that we can always find Gorenstein sets of points containing our general points $X_{e,t}$ and show how liaison is used to build larger sets with the desired resolution. This is the basis for our induction on $e$ to prove the MRC.

%%%%%%%%%%%%%%%%%%%%%%%%%%%%%%

\subsection{Links of type 1} \label{Links of type 1}

Let $S$ be a surface of  degree $d$ (either even or odd) and let $X_{e-2,t} \subset S$ with $e \geq d+1$. Assume that $X_{e-2,t}$ satisfies the MRC. The $h$-vector of $X_{e-2,t}$ is
\[
 (1,3,6,\dots, h_S (e-3),t)
\]
and the minimal free resolution of $X_{e-2,t}$ is
\[
0 \rightarrow
\begin{array}{c}
R(-e-1)^{c_2} \\
\oplus \\
R(-e)^{c_1}
\end{array}
\rightarrow
\begin{array}{c}
R(-e)^{b_2} \\
\oplus \\
R(-e+1)^{b_1}
\end{array}
\rightarrow
\begin{array}{c}
R(-e+1)^{a_2} \\
\oplus \\
R(-e+2)^{a_1}\\
\oplus \\
R(-d)
\end{array}
\rightarrow I_{X_{e-2,t}} \rightarrow 0
\]
where $a_2 \cdot b_1 = 0$, $b_2 \cdot c_1 = 0$, $a_1 = h_S(e-2)-t$ and  $c_2 = t$. Notice that necessarily we have $1 \leq t \leq h_S(e-2)-1$ by definition of $e$.

Let $G$ be an arithmetically Gorenstein set of points containing $X_{e-2,t}$ that is relatively compressed of socle degree $2e-2$.  $G$ links $X_{e-2,t}$ to a residual set $Z$ with $h$-vector
\[
(1,3,6,\dots, h_S(e-1), h_S(e-2)-t).
\]
We will call this a {\em link of type 1}. The minimal free resolution of $I_G$ is
\begin{equation} \label{type 1 resolution}
0 \rightarrow R(-2e-1) \rightarrow
\begin{array}{c}
R(-2e-1+d) \\
\oplus \\
R(-e-1)^{2d}
\end{array}
\rightarrow
\begin{array}{c}
R(-e)^{2d}\\
\oplus \\
R(-d)
\end{array}
\rightarrow I_G \rightarrow 0.
\end{equation}
Splitting off the two copies of $R(-d)$, the mapping cone (see \cite{Wei}, \cite{PS}) gives the following minimal free resolution for $I_Z$:
\[
0 \rightarrow
\begin{array}{c}
R(-e-3)^{a_1} \\
\oplus \\
R(-e-2)^{a_2}
\end{array}
\rightarrow
\begin{array}{c}
R(-e-2)^{b_1} \\
\oplus \\
R(-e-1)^{b_2+2d}
\end{array}
\rightarrow
\begin{array}{c}
R(-e-1)^{c_1} \\
\oplus \\
R(-e)^{2d+c_2}\\
\oplus \\
R(-d)
\end{array}
\rightarrow I_{Z} \rightarrow 0
\]
We make the following observations.

\begin{enumerate}

\item The condition $a_2 \cdot b_1 = 0$ guarantees that there are no redundant copies of $R(-e-2)$.

\item If $c_1 = 0$ then there are no redundant copies of $R(-e-1)$. Equivalently, this holds if $t \geq m_2(e-2)$.

\item The socle degree of $Z$ is $e$ and the surplus is $s = h_S(e-2) -t = a_1$.

\end{enumerate}

We have shown:

\begin{proposition} \label{type 1}
Assume that $X_{e-2,t}$ has the minimal free resolution predicted by the MRC. 
Assume that a relatively compressed arithmetically Gorenstein set of points $G$ with socle degree $2e-2$ can be found containing $X_{e-2,t}$. Assume $e \geq d+1$. If $h_S(e-2)>t \geq m_2(e-2)$ then
$G$ links $X_{e-2,t}$ to a set of points $Z$ with socle degree $e$, relatively compressed $h$-vector, and having the minimal free resolution predicted by the MRC. The surplus, $s$, of $Z$ satisfies
\[
1 \leq s = h_S(e-2)-t  \leq \frac{2}{3} \cdot h_S(e) - \frac{5d}{3}.
\]
\end{proposition}

\begin{proof}
The fact that $Z$ has the minimal free resolution predicted by the MRC is what we proved before the statement of this proposition. For the rest, we have
\[
\begin{array}{rcl}
1 \ \leq \ s & = & h_S(e-2) - t \\
& \leq & h_S(e-2) - \frac{1}{3} h_S(e-3) \\
& = & \displaystyle \frac{3h_S(e-2) - h_S(e-3)}{3} \\
& = & \displaystyle \frac{2 h_S(e-2) -d}{3} \\
& = & \displaystyle \frac{2 (h_S(e) - 2d) -d}{3}
\end{array}
\]
from which the result follows. The only comment is that we used the definition of $m_2(e-2)$, Proposition \ref{intervals} and the fact that $h_S(e-2) - h_S(e-3) = d$.
\end{proof}

%%%%%%%%%%%%%%%%%%%%%%%%%%%%%%

\subsection{Links of type 2} \label{Links of type 2}

As explained in Subsection \ref{problem with odd degree surface}, we will now assume that $d$ is even.
Let $S$ be a surface of even degree $d$ and let $X_{e-1,t} \subset S$ with $e \geq d+1$. 
%As we will explain shortly, we will also assume that $S$ is defined by the Pfaffian of a $(d \times d)$ skew-symmetric matrix of quadrics.

Assume that $X_{e-1,t}$ satisfies the MRC. The $h$-vector of $X_{e-1,t}$ is
\[
 (1,3,6,\dots, h_S (e-2),t)
\]
and the minimal free resolution of $X_{e-1,t}$ is
\[
0 \rightarrow
\begin{array}{c}
R(-e-2)^{c_2} \\
\oplus \\
R(-e-1)^{c_1}
\end{array}
\rightarrow
\begin{array}{c}
R(-e-1)^{b_2} \\
\oplus \\
R(-e)^{b_1}
\end{array}
\rightarrow
\begin{array}{c}
R(-e)^{a_2} \\
\oplus \\
R(-e+1)^{a_1}\\
\oplus \\
R(-d)
\end{array}
\rightarrow I_{X_{e-1,t}} \rightarrow 0
\]
where $a_2 \cdot b_1 = 0$, $b_2 \cdot c_1 = 0$, $a_1 = h_S(e-1)-t$, and $c_2 = t$.

Let $G$ be an arithmetically Gorenstein set of points on $S$ that is relatively compressed of socle degree $2e-1$, having $d+1$ minimal generators (recall that $d$ is even).
% By \cite{BE} if such $G$ exists on $S$ then $S$ is defined by the Pfaffian of a $(d \times d)$ skew-symmetric matrix of quadrics (since $S$ is one of the minimal generators of $I_G$). Conversely, given such a Pfaffian defining $S$, it can be extended by adding a sufficiently general row of forms of degree $e+2-d$ and a corresponding column to produce a skew-symmetric $(d+1) \times (d+1)$ matrix whose $(d \times d)$ Pfaffians generate $I_G$.
Ignoring for a moment the question of containing $X_{e-1,t}$, the minimal free resolution of  such $I_G$ is
\begin{equation} \label{type 2 resolution}
0 \rightarrow R(-2e-2) \rightarrow
\begin{array}{c}
R(-2e-2+d) \\
\oplus \\
R(-e-2)^{d}
\end{array}
\rightarrow
\begin{array}{c}
R(-e)^{d}\\
\oplus \\
R(-d)
\end{array}
\rightarrow I_G \rightarrow 0.
\end{equation}

If such $G$ exists containing $X_{e-1,t}$, we will call this a {\em link of type 2}. Then $G$ links $X_{e-1,t}$ to a residual set $Z$ with $h$-vector
\[
(1,3,6,\dots, h_S(e-1), h_S(e-1)-t).
\]
Splitting off the two copies of $R(-d)$, the mapping cone gives the following minimal free resolution for $I_Z$:
\[
0 \rightarrow
\begin{array}{c}
R(-e-3)^{a_1} \\
\oplus \\
R(-e-2)^{a_2}
\end{array}
\rightarrow
\begin{array}{c}
R(-e-2)^{d+b_1} \\
\oplus \\
R(-e-1)^{b_2}
\end{array}
\rightarrow
\begin{array}{c}
R(-e-1)^{c_1} \\
\oplus \\
R(-e)^{d+c_2}\\
\oplus \\
R(-d)
\end{array}
\rightarrow I_{Z} \rightarrow 0
\]
We make the following observations.

\begin{enumerate}

\item The condition $b_2 \cdot c_1 = 0$ guarantees that there are no redundant copies of $R(-e-1)$.

\item If $a_2 = 0$ then there are no redundant copies of $R(-e-2)$ in the latter minimal free resolution. This happens when $I_{X_{e-1,t}}$ (in the first resolution) has no minimal generator of degree $e$. Equivalently, since we assumed that $X_{e-1,t}$ satisfies the MRC, this holds if $t \leq m_3(e-1)$.

\item If $b_1 =0$, then we again get no redundant copies of $R(-e-2)$ provided we can split off all the degree $e$ minimal generators of $I_G$. However, in general this is not possible.

\item The socle degree of $Z$ is $e$ and the surplus is $s = h_S(e-1) -t$.

\end{enumerate}

We have shown:

\begin{proposition} \label{type 2}
Assume that $X_{e-1,t}$ has the minimal free resolution predicted by the MRC. 
Assume that $d$ is even and that $e \geq d+1$.
Assume that a relatively compressed arithmetically Gorenstein set of points $G$ with socle degree $2e-1$ and resolution (\ref{type 2 resolution}) can be found containing $X_{e-1,t}$. If $t \leq m_3(e-1)$ then $G$ links $X_{e-1,t}$ to a set of points $Z$ with socle degree~$e$, having a relatively compressed $h$-vector, and having the minimal free resolution predicted by the MRC. The surplus, $s$, of $Z$ satisfies
\[
h_S(e-1) - 1 \geq s = h_S(e-1)-t  \geq  \frac{h_S(e)}{3} .
\]
\end{proposition}

\begin{proof}
The fact that $Z$ has the minimal free resolution predicted by the MRC is what we proved before the statement of this proposition. For the rest, 
\[
\begin{array}{rcl}
s & = & h_S(e-1)-t \\
& \geq &  h_S(e-1) - [h_S(e-1) - \frac{1}{3} h_S(e) ]
\end{array}
\]
from which the result follows.
The only comment is that we used the definition of $m_3(e-1)$ and Proposition \ref{intervals}.
\end{proof}

%%%%%%%%%%%%%%%%%%%%%%%%%%%%%%

\subsection{Consequences}

\begin{lemma} \label{lemma1}
Assume that $d \geq 4$, $e\ge d+1$ and $(d,e) \neq (4,5)$.
\[
\frac{2}{3} \cdot h_S(e) - \frac{5d}{3} \geq \frac{h_S(e)}{3}  .
\]
\end{lemma}

\begin{proof}
Notice that $e \geq \frac{d+7}{2}$, from which it follows that $10 \leq 2e-d+3$. Then
\[
\frac{5}{3} \leq \frac{2e-d+3}{6}
\]
so
\[
\frac{5}{3} d \leq \frac{2de - d^2 + 3d}{6} = \frac{2de - [d^2-3d+2]+2}{6} = \frac{1}{3} \left [ de - \binom{d-1}{2} +1 \right ] = \frac{1}{3} h_S(e).
\]
This means
\[
\frac{2}{3} h_S(e) - \frac{5}{3} d \geq \frac{1}{3} h_S(e)
\]
as desired.
\end{proof}

\begin{lemma} \label{lemma2}
\[
h_S (e-1) -1 \geq m_4(e) .
\]
\end{lemma}

\begin{proof}
We have to show that
\[
3 \left [ h_S(e) - (h_S(e-1)-1) \right ] \leq h_S(e+1).
\]
Using Remark \ref{formulas} and the fact that $h_S(e) - h_S(e-1) = d$, we have to show that
\[
3d+3 \leq d(e+1) - \binom{d-1}{2} + 1 .
\]
This reduces to showing that
\[
\frac{d-3}{2} + \frac{3}{d} \leq e-2,
\]
which is clearly true.
\end{proof}

Combining Proposition \ref{type 1}, Proposition \ref{type 2}, Lemma \ref{lemma1} and Lemma \ref{lemma2}, we obtain the following.

\begin{theorem} \label{MRC theorem}
Assume that there is some irreducible surface $S$ of degree $d$ such that

\begin{itemize}
\item[(a)] $d$ is even;

\item[(b)] for $e \leq d$, $X_{e,t}$ satisfies the MRC for any $t$;

\item[(c)] for $e \geq d+1$ and $t$ satisfying the conditions of Propositions \ref{type 1} and \ref{type 2}, Gorenstein links of types 1 and 2 can be made for $X_{e,t}$.

\end{itemize}

\noindent Then the MRC holds on $S$, hence on a general surface  of degree $d$.
\end{theorem}

\begin{proof}
The proof is by induction on $e$. For $e \leq d$, assumption (b) begins the induction.

We first consider the case $d=4, e=5$. In this case $m_2(e-2) = 2$ and $m_3(e-1) = 8$. Proposition \ref{type 1} shows that $X_{e,s}$ has the desired resolution for $1 \leq s \leq 10-t \leq 8$. Proposition \ref{type 2} shows that $X_{e,s}$ has the desired resolution for $13 \geq s = 14-t \geq 6$. Hence the result holds for $d=4, e=5$. From now on in this proof we assume that $(d,e) \neq (4,5)$.

Thanks to Proposition \ref{type 1}, Proposition \ref{type 2}, Lemma \ref{lemma1} and semicontinuity, the MRC is true for $X_{e,s}$ for
\[
1 \leq s \leq h_S(e-1)-t.
\]
Thanks to Lemma \ref{lemma2} and Proposition \ref{intervals}, this is enough to deduce the truth of the MRC for socle degree $e$.
\end{proof}

%%%%%%%%%%%%%%%%%%%%%%%%%%%%%%

\subsection{Surfaces of odd degree} \label{problem with odd degree surface}

In using links of type 2 we were forced to assume that $d$ is even. In this subsection we discuss the problems with surfaces of odd degree.

\begin{example}
Let $d=5$, $e=6$ and consider the set $X_{6,t}$. Its $h$-vector is $(1,3,6,10,15,20,t)$.
%Assume that $t < m_1(6)$.
A relatively compressed arithmetically Gorenstein set of points $G$ with socle degree 13 and containing $X_{6,t}$ links $X_{6,t}$ to a set $Z$ with $h$-vector $(1,3,6,10,15,20,25, 25-t)$ so one could hope that the latter set is ``general enough." But consider the minimal free resolutions
\[
0 \rightarrow
\begin{array}{c}
R(-9)^{c_2} \\
\oplus \\
R(-8)^{c_1}
\end{array}
\rightarrow
\begin{array}{c}
R(-8)^{b_2} \\
\oplus \\
R(-7)^{b_1}
\end{array}
\rightarrow
\begin{array}{c}
R(-7)^{a_2} \\
\oplus \\
R(-6)^{a_1}\\
\oplus \\
R(-5)
\end{array}
\rightarrow I_{X_{6,t}} \rightarrow 0
\]
where $c_2 = t$,
%$c_1 \neq 0$ (since $t < m_1(6)$),
$b_2 \cdot c_1 = 0$, $a_2 \cdot b_1 = 0$, $a_1 = 25-t$, and
\[
0 \rightarrow R(-16) \rightarrow
\begin{array}{c}
R(-11) \\
\oplus \\
R(-9)^{5}\\
\oplus \\
R(-8)
\end{array}
\rightarrow
\begin{array}{c}
R(-8) \\
\oplus \\
R(-7)^{5}\\
\oplus \\
R(-5)
\end{array}
\rightarrow I_G \rightarrow 0
\]
(even supposing that $G$ could be found with only one generator of degree 8, which is not obvious).
The free resolution for $Z$ then has the form
\[
0 \rightarrow
\begin{array}{c}
R(-10)^{a_1} \\
\oplus \\
R(-9)^{a_2}
\end{array}
\rightarrow
\begin{array}{c}
R(-9)^{5+b_1} \\
\oplus \\
R(-8)^{1+b_2}
\end{array}
\rightarrow
\begin{array}{c}
R(-8)^{1+c_1} \\
\oplus \\
R(-7)^{5+c_2}\\
\oplus \\
R(-5)
\end{array}
\rightarrow I_{Z} \rightarrow 0
\]
Assume that all possible splitting has been performed. The copy of $R(-8)$ coming from the minimal generators of $I_G$ does not split with anything, so the summand $R(-8)^{1+b_2}$ does not reduce to zero in the minimal free resolution of $I_Z$.
On the other hand, it is not at all clear that it is possible to split off the $R(-8)$ corresponding to a first syzygy of $I_G$, and even if it were possible, if $t < m_1(6)$ then $c_1 > 0$ and the summands $R(-8)^{c_1}$ do not split. Thus the latter minimal free resolution is probably not of the desired type in general, and certainly not if $t < m_1(6)$.
\end{example}

%%%%%%%%%%%%%%%%%%%%%%%%%%%%%%%%%%%%%%%%%%%%%%%%%%%%%%%%%%%%%%

\section{Finding Gorenstein links}

In the previous section we showed what we can say about the MRC assuming that Gorenstein links of a certain kind could be found. In this section we address this problem.

There are two approaches that suggest themselves. One is to find a special surface $S$ of degree $d$ where we have enough control to construct many Gorenstein sets of points, and try to inductively produce sets $X$ on this surface, for any socle degree and surplus, having the desired resolution. Then consider the incidence variety $I$ inside $Hilb^m(\mathbb P^3) \times A $, where $A \subset H^0(\mathcal O_{\mathbb P^3}(4))$ is the open subset consisting of reduced, irreducible quartic surfaces. Since any fiber of either projection from $I$ is irreducible, $I$ is irreducible as well. Thus semicontinuity gives the result for a general set of points of the same cardinality on a general surface of degree $d$. 

The other approach is to consider general $S$ and use dimension counts to force the existence of the two types of Gorenstein links for $X_{e,t}$. An earlier version of this paper took the latter approach, but the referee pointed out one step that we were not able to justify (although we believe it to be correct). Now instead we solve this problem using the former approach. 

In the situation where $e \geq d+1$, $S$ is the unique surface of degree $d$ containing $X_{e,t}$, and more importantly it is the unique surface of degree $d$ containing the Gorenstein sets $G$ that we want to use for our links. Hence it corresponds to a minimal generator of $I_G$.

We now turn to the MRC in the case of quartics ($d=4$).

\begin{theorem} \label{MRC on quartics}
Let $S$ be a general surface of degree 4 in $\mathbb P^3$. Then the MRC holds for points on $S$.
\end{theorem}

\begin{proof}

%By the main result of \cite{BG}, condition (b) of Theorem \ref{MRC theorem} is true. Thus we only have to show that condition (c) is true. Our strategy will be to work not on a truly general surface of degree 4, but rather a ``suitably" general surface. Then by semicontinuity, the result will follow. 

As mentioned above, we will find a specific quartic on which the MRC holds, and conclude by semicontinuity. To find a quartic surface $S$ suitable for our needs, we make the following construction. Let $Z$ be a set of 16 general points in $\mathbb P^3$ and write $Z = Z_1 \cup Z_2$ where $|Z_1| = 12$ and $|Z_2| = 4$. 

We first claim that we can find a smooth arithmetically Cohen-Macaulay curve $C_1$ of degree 6 and genus 3 containing $Z_1$. Indeed, thanks to \cite{E} Example 2 (page 430), or more precisely \cite{E} Th\'eor\`eme 2, the dimension of the Hilbert scheme $H_{6,3}$ of curves in $\mathbb P^3$ of degree 6 and genus 3 (the general one of which is arithmetically Cohen-Macaulay with minimal free resolution
\[
0 \rightarrow R(-4)^3 \rightarrow R(-3)^4 \rightarrow I \rightarrow 0 
\]
and is smooth) has dimension 24. Then \cite{P} Proposition 5.11.bis and Corollaire 5.11 show that there is a smooth, arithmetically Cohen-Macaulay curve $C_1$ of degree 6 and genus 3 containing $Z_1$, and \cite{P} Proposition 2.1 shows that 12 is the greatest number of general points for which this is true.

Let $Y \subset Z_1$ be a choice of 8 of the points. Since $\dim [R]_2 = 10$, there is a pencil of quadric surfaces containing $Y$, defining a complete intersection curve $C_2$ of type $(2,2)$ containing $Y$. Both $C_1$ and $C_2$ are smooth, and they intersect at the eight points $Y$. The union $C_1 \cup C_2$ is arithmetically Cohen-Macaulay, since $I_{C_1} + I_{C_2}$ is the saturated ideal of $Y$; indeed, this follows from the exact sequence
\[
0 \rightarrow I_{C_1 \cup C_2} \rightarrow I_{C_1} \oplus I_{C_2} \rightarrow I_{C_1} + I_{C_2} \rightarrow 0.
\]
One checks that the union $C_1 \cup C_2$ lies on a 5-dimensional (vector space) family of quartic surfaces.  Thus we can find a quartic surface $S$ containing $C_1 \cup C_2 \cup Z_2$. Notice that $Z_1$ also lies on $S$ since $S$ contains $C_1$. In other words, we can find a quartic surface containing a suitable union $C_1 \cup C_2$ and also containing 16 general points (any subset of which has the expected resolution thanks to \cite{BG}).

We first focus on part (c) of Theorem \ref{MRC theorem}, which is the inductive step, and then turn to part (b) which begins the induction.

We first consider Gorenstein links of type 1 as described in Subsection \ref{Links of type 1} and specifically in Proposition \ref{type 1}. In particular, we will be considering general sets of points $X_{e-2,t}$ of socle degree $e-2$ satisfying $m_2(e-2) \leq t < h_S(e-2)$ on the quartic surface $S$ constructed above. We assume $e \geq 5$.

We want to show that a Gorenstein link of type 1 exists. That is, we want a relatively compressed Gorenstein set of points, $G$, on $S$ with socle degree $2e-2$ and containing $X_{e-2,t}$. The $h$-vectors of $X_{e-2,t}$ and of the desired $G$ are as follows:

{\footnotesize
\begin{center}
\begin{equation} \label{type 1 hvector}
\begin{array}{c|ccccccccccccccccccc}
\hbox{deg} & 0 & 1 & 2 & 3 & 4 & 5 & \dots & e-3 & e-2 & e-1 & e & \dots & 2e-4 & 2e-3 & 2e-2 \\ \hline \\
G & 1 & 3 & 6 & 10 & 14 & 18 & \dots & 4e-14 & 4e-10 & 4e-6 & 4e-10 & \dots & 6 & 3 & 1 \\ \\
X_{e-2,t} & 1 & 3 & 6 & 10 & 14 & 18 & \dots & 4e-14 & t & - & - & \dots & - & - & -

\end{array}
\end{equation}
\end{center}
}
A standard calculation gives
\[
| X_{e-2,t} | = 2e^2 - 12e + 20 + t,
\]
and 
\[
\deg G = |G| = 8 \binom{e-1}{2} + 6.
\]
To produce $G$, we begin with the smooth arithmetically Cohen-Macaulay curve $C_1$ of degree 6 and genus 3 mentioned above. Notice that 
\[
\dim [I_{C_1}]_m = \binom{m+3}{3} - h^0(\mathcal O_{C_1}(m)) = \binom{m+3}{3} - (6m - 3 + 1) .
\]
%In particular, when $m=4$ we have $\dim [I_{C_1}]_4 = 13$. Let $S$ be a general quartic surface containing $C_1$. 

%The above calculation when $m=4$ shows that we can even assume that $S$ contains up to 12 general points of $\mathbb P^3$, which we know from \cite{BG} have the correct resolution. 
%We also know that $e \geq 5$. 
%When $e=5$, $|X_{e-2,t}| \leq 10$. This is enough to guarantee that we can use Proposition \ref{type 1} as long as suitable $G$ can be found, which is our main task here.
In particular,
\[
\dim [I_{C_1}]_e = \binom{e+3}{3} - 6e + 2.
\]
Let $F$ be a form of degree 4 defining the surface $S$. We also note that in $[I_{C_1}]_e$ the subspace given by multiples of $F$ has dimension $\binom{e-1}{3}$.

If we link $C_1$ using a  complete intersection of type $(4,e)$, defined by $S$ and a general element of $[I_{C_1}]_e$, the residual is a smooth curve $C$ of degree $4e-6$ and (one computes) genus $g = 4 \binom{e-2}{2} + 4e -9$. Let $K_C$ be a canonical divisor on $C$ and $H_C$ a general hyperplane section. A general element of the linear system
\[
|(2e-3)H_C - K_C|
\]
is arithmetically Gorenstein with the $h$-vector given in (\ref{type 1 hvector}) and the minimal free resolution given in (\ref{type 1 resolution}). Indeed, this is a consequence of the construction given in \cite{KMMNP} Lemma 5.4, taking $\ell = 0$ and $d = 2e-3$.

We now claim that this link can even be performed if we further require the element of degree $e$ to vanish at $|X_{e-2,t}|$ general points on $S$. Notice that since our general points lie on $S$, we have to account for the surfaces containing $S$ as a component, i.e. we have to subtract $\binom{e-1}{3}$ as well. Then we have
\[
\begin{array}{rcl}
\displaystyle \binom{e+3}{3} - 6e + 2 - \binom{e-1}{3} - \left [ 2e^2 - 12e + 20 + t \right ] & = & 6e-16-t \\ 
& \geq & 6e-16 - h_S(e-2) \\
& =  & 6e-16 - (4e-10) \\
& > & 0.
\end{array}
\]
Thus given $|X_{e-2,t}|$ general points of $S$, we can find a smooth curve $C$ of degree $4e-6$ and genus $4 \binom{e-2}{2} + 4e-9$ containing those points.

On $C$ we use Riemann-Roch to compute the dimension of the linear system of arithmetically Gorenstein sets of points:
\[
\begin{array}{rcl}
\dim | (2e-3)H_C - K_C | & = & - 1 + h^0(\mathcal O_C ((2e-3)H_C - K_C)) \\ \\
& = & -1 + (\deg G - g_C + 1 + h^1 (\mathcal O_C ((2e-3)H_C - K_C)) \\ \\
& \geq & \displaystyle -1 + \left [ 8 \binom{e-1}{2} + 6 \right ] - \left [ 4 \binom{e-2}{2} + 4e - 9 \right ] + 1 \\ \\
& = & \displaystyle \frac{4e^2 - 12e + 22}{2}.
\end{array}
\]

Finally, we have to check that this linear system is big enough that we can find an element containing the subset $X_{e-2,t}$ of $C$. 
\[
\frac{4e^2 - 12e + 22}{2} - (2e^2 - 12e + 20 + t) = 6e -9 -t > 6e-9 - (4e-10) = 2e + 1 > 0.
\]
Because the minimal free resolution of $G$ is the desired one, we have shown part (c) of Theorem \ref{MRC theorem} for links of type 1. 

%%%%%%%%%%%%%%%%%%%%%%%%%%%%%%%%%%%%%%%%%%%%%%%%%%%%%%%%%%%%
%%%%%%%%%%%%%%%%%%%%%%%%%%%%%%%%%%%%%%%%%%%%%%%%%%%%%%%%%%%%

We next consider Gorenstein links of type 2 as described in Subsection \ref{Links of type 2}, so in particular we will be considering general sets of points $X_{e-1,t}$ of socle degree $e-1$ satisfying $ t \leq m_3(e-1)$ on a suitable quartic surface $S$. Again, we will handle the case $e \leq 4$ later and for now assume $e \geq 5$.

We want to show that a Gorenstein link of type 2 exists on $S$. That is, we want a relatively compressed Gorenstein set of points, $G$, with socle degree $2e-1$, containing $X_{e-1,t}$ and having the correct minimal free resolution. The $h$-vectors of $X_{e-1,t}$ and of the desired $G$ are as follows:

{\footnotesize
\begin{center}
\begin{equation} \label{type 2 hvector}
\begin{array}{c|ccccccccccccccccccc}
\hbox{deg} & 0 & 1 & 2 & 3 & 4 & 5 & \dots & e-2 & e-1 & e & e+1 & \dots & 2e-3 & 2e-2 & 2e-1 \\ \hline \\
G & 1 & 3 & 6 & 10 & 14 & 18 & \dots & 4e-10 & 4e-6 & 4e-6 & 4e-10 & \dots & 6 & 3 & 1 \\ \\
X_{e-1,t} & 1 & 3 & 6 & 10 & 14 & 18 & \dots & 4e-10 & t & - & - & \dots & - & - & -

\end{array}
\end{equation}
\end{center}
}
A standard calculation gives
\[
| X_{e-1,t} | = 2e^2 - 8e + 10 + t,
\]
and 
\[
\deg G = |G| = 4e^2 - 8e +8
\]
To produce $G$, we begin with our smooth complete intersection curve $C_2$ of type $(2,2)$, degree 4 and genus 1 on $S$. Notice that 
\[
\dim [I_{C_2}]_m = \binom{m+3}{3} - h^0(\mathcal O_{C_2}(m)) = \binom{m+3}{3} - (4m-1+1) .
\]
In particular, when $m=4$ we have $\dim [I_{C_2}]_4 = 19$. 

%The above calculation when $m=4$ shows that we can even assume that $S$ contains up to 18 general points of $\mathbb P^3$, which we know from \cite{BG} have the correct resolution. This is not quite enough to begin the induction, but we will return to this issue shortly.

We also note that 
\[
\dim [I_{C_2}]_e = \binom{e+3}{3} - 4e
\]
Let $F$ be a form of degree 4 defining the surface $S$. We also note that in $[I_{C_2}]_e$ the subspace given by multiples of $F$ has dimension $\binom{e-1}{3}$.

If we link $C_2$ using a  complete intersection of type $(4,e)$, defined by $S$ and a general element of $[I_{C_2}]_e$, the residual is a smooth curve $C$ of degree $4e-4$  and (one computes) genus $g = 2e^2-4e+1$. Let $K_C$ be a canonical divisor on $C$ and $H_C$ a general hyperplane section. A general element of the linear system
\[
|(2e-2)H_C - K_C|
\]
is arithmetically Gorenstein with the $h$-vector given in (\ref{type 2 hvector}) and the minimal free resolution given in (\ref{type 2 resolution}). Indeed, this is a consequence of the construction given in \cite{KMMNP} Lemma 5.4, taking $\ell = 0$ and $d = 2e-2$.

We now claim that this link can even be performed if we further require the element of degree $e$ to vanish at $|X_{e-1,t}|$ general points on $S$. Indeed,
\[
\begin{array}{rcl}
\displaystyle \binom{e+3}{3} - 4e - \binom{e-1}{3} - \left [ 2e^2 - 8e + 10 + t \right ] & = & 4e-8-t\\ 
& \geq & 4e-8 - m_3(e-1) \\
& =  & \displaystyle 4e-8 - \left \lfloor \frac{8e-16}{3} \right \rfloor  \\
& > & 0
\end{array}
\]
(since $e \geq 5$).
Thus given $|X_{e-1,t}|$ general points of $S$, we can find a smooth curve $C$ of degree $4e-4$ and genus $2e^2-4e+1$ containing those points.

On $C$ we use Riemann-Roch to compute the dimension of the corresponding linear system of arithmetically Gorenstein sets of points:
\[
\begin{array}{rcl}
\dim | (2e-2)H_C - K_C | & = & - 1 + h^0(\mathcal O_C ((2e-2)H_C - K_C)) \\ \\
& = & -1 + (\deg G - g_C + 1 + h^1 (\mathcal O_C ((2e-2)H_C - K_C)) \\ \\
& \geq & \displaystyle -1 + \left [ 4e^2-8e+8 \right ] - \left [ 2e^2-4e+1 \right ] + 1 \\ \\
& = & 2e^2 -4e+7
\end{array}
\]

Now we have to check that this linear system is big enough that we can find an element containing the subset $X_{e-2,t}$ of $C$. 
\[
2e^2 - 4e+7 - (2e^2 - 8e+10 + t) = 4e -3 -t \geq 4e-3 - \left \lfloor \frac{8e-16}{3} \right \rfloor > 0.
\]
Because the minimal free resolution of $G$ is the desired one, we have shown part (c) of Theorem \ref{MRC theorem} also for links of type 2. 

What is left is to show that a general set of points of socle degree $\leq 4$ on $S$ has the correct resolution. 
Notice that on $S$, 
\[
\begin{array}{cccccccc}
m_1(2) & = & 1 \\
m_2(2) & = & m_3(2) & = &  2 \\
m_4(2) & = & 3 \\
m_1(3) & = & m_2(3) & = & 2 \\
m_3(3) & = & 5 \\
m_4(3) & = & 6 \\
m_1(4) & = & 3 \\
m_2(4) & = & 4 \\
m_3(4) & = & m_4(4) & = & 8
\end{array}
\]
so we only have to check the critical values of 5, 6, 7, 12, 15, 16, 23, 24, or 28 general points (see Proposition \ref{intervals}) to start the induction. 

Our construction of $S$ included the fact that up to 16 points with the expected resolution lie on $S$, so by semicontinuity we know that the MRC holds for the critical values up to 16. We only have to check 23, 24 and 28.

We consider 23 points on $S$. In the construction of $S$ we considered the Cohen-Macaulay union $C_1 \cup C_2$. Using the smooth surface $S$ and a general element in $[I_{C_1 \cup C_2}]_5$, we obtain a smooth Cohen-Macaulay residual curve $C$ of degree 10 and genus 11. A general element $G$ of the linear system $|5 H_C - K_C|$ is a reduced, arithmetically Gorenstein set of points with $h$-vector $(1,3,6,10,6,3,1)$ and minimal free resolution
\[
0 \rightarrow R(-9) \rightarrow R(-5)^9 \rightarrow R(-4)^9 \rightarrow I_G \rightarrow 0.
\]
As above, the linear system is big enough to contain 7 general points, and the residual is a set of 23 points with the desired resolution. To construct 24 points with the right resolution, the procedure is the same but we just use 6 points.

To handle 28 points we use the same construction as in the above proof, linking from a curve $C_2$ that is the complete intersection of two quadrics, using the complete intersection of $S$ and a general quartic in $I_{C_2}$, to get a residual curve of degree 12, genus 17 and $h$-vector $(1,2,3,4,2)$. Then proceeding as in the case of type 2 above, we consider the linear system $|6H_C - K_C|$, a general element of which is a reduced, arithmetically Gorenstein set of points with $h$-vector $(1,3,6,10, 10, 6,3,1)$ and minimal free resolution
\[
0 \rightarrow R(-10) \rightarrow R(-6)^5 \rightarrow R(-4)^5 \rightarrow I_G \rightarrow 0.
\]
Computing dimensions as before allows us to find such a set of points containing 12 general ones, and the residual is the desired set of 28 points.
%What we have done is to construct on the computer algebra programs CoCoA \cite{cocoa} and Macaulay2 \cite{macaulay} a surface of the kind constructed in this proof and found a suitable set of points whose resolution was the desired one. The induction required us to account for all cardinalities from 1 to 34. But in fact, cardinalities 1 to 5 are trivial (the latter is arithmetically Gorenstein), and since 
%Using ad hoc methods, we found such examples over $\mathbb Z/(31991 \mathbb Z)$ and over $\mathbb Q$.
\end{proof}

%\begin{remark}
%For surfaces of higher degree, the dimension counts do not work as nicely. For instance, when $d=6$ the links of type 1 would allow for at least a partial result (MRC for socle degree $e \leq 13$). However, links of type 2 are more delicate, and we were not able to obtain any reasonable result with this approach.
%\end{remark}

\begin{remark} \label{higher degree remark}
For surfaces of even higher degree, arguing as above one shows that the links of type 1 continue to exist. However, unfortunately the dimension counts do not force the existence of the links of type 2 when $d>4$. So in order to prove the MRC for surfaces
of higher degree, another approach may be required.
\end{remark}

\end{document}